\documentclass[11pt]{amsart}
\usepackage{graphicx}
\usepackage{amsmath,amsthm,amssymb,enumerate}
\usepackage{euscript,mathrsfs}
\usepackage[citecolor=blue,colorlinks=true]{hyperref}
\usepackage[left=3cm,right=3cm,top=4.5cm,bottom=4.5cm]{geometry}
\usepackage{enumitem}
\setenumerate{label={\rm (\alph{*})}}
\usepackage{color}
\catcode`\@=11 \@addtoreset{equation}{section}

\catcode`\@=12

\allowdisplaybreaks

\newtheorem{Theorem}{Theorem}[section]
\newtheorem{Proposition}[Theorem]{Proposition}
\newtheorem{Lemma}[Theorem]{Lemma}
\newtheorem{Corollary}[Theorem]{Corollary}

\theoremstyle{definition}
\newtheorem{Definition}[Theorem]{Definition}

\newtheorem{Remark}[Theorem]{Remark}

\newcommand{\bTheorem}[1]{
\begin{Theorem} \label{T#1} }
\newcommand{\eT}{\end{Theorem}}

\newcommand{\bProposition}[1]{
\begin{Proposition} \label{P#1}}
\newcommand{\eP}{\end{Proposition}}

\newcommand{\bLemma}[1]{
\begin{Lemma} \label{L#1} }
\newcommand{\eL}{\end{Lemma}}

\newcommand{\bCorollary}[1]{
\begin{Corollary} \label{C#1} }
\newcommand{\eC}{\end{Corollary}}

\newcommand{\bRemark}[1]{
\begin{Remark} \label{R#1} }
\newcommand{\eR}{\end{Remark}}

\newcommand{\bDefinition}[1]{
\begin{Definition} \label{D#1} }
\newcommand{\eD}{\end{Definition}}

\newcommand{\Q}{T^N}

\newcommand{\dif}{\mathrm{d}}
\newcommand{\Del}{\Delta_x}

\newcommand{\vme}{\vm_\ep}

\newcommand{\tvm}{{\tilde{\vc{m}}}}
\newcommand{\bfphi}{\boldsymbol{\varphi}}

\newcommand{\bFormula}[1]{
\begin{equation} \label{#1}}
\newcommand{\eF}{\end{equation}}

\newcommand{\Ov}[1]{\overline{#1}}

\newcommand{\aleq}{\lesssim}

\newcommand{\vr}{\varrho}
\newcommand{\vre}{\vr_\ep}

\newcommand{\vue}{\vu_\ep}
\newcommand{\tvr}{\tilde{\varrho}}

\newcommand{\vu}{\vc{u}}
\newcommand{\vm}{\vc{m}}

\newcommand{\vc}[1]{{\bf #1}}

\newcommand{\Div}{{\rm div}_x}
\newcommand{\Grad}{\nabla_x}

\newcommand{\dx}{\,{\rm d} {x}}

\newcommand{\dt}{\,{\rm d} t }

\newcommand{\intQ}[1]{\int_{{\Q}} #1 \ \dx}

\newcommand{\D}{{\rm d}}
\newcommand{\ep}{\varepsilon}

\definecolor{Cgrey}{rgb}{0.85,0.85,0.85}
\definecolor{Cblue}{rgb}{0.50,0.85,0.85}
\definecolor{Cred}{rgb}{1,0,0}
\definecolor{fancy}{rgb}{0.10,0.85,0.10}

\newcommand\Cbox[2]{%
    \newbox\contentbox%
    \newbox\bkgdbox%
    \setbox\contentbox\hbox to \hsize{%
        \vtop{
            \kern\columnsep
            \hbox to \hsize{%
                \kern\columnsep%
                \advance\hsize by -2\columnsep%
                \setlength{\textwidth}{\hsize}%
                \vbox{
                    \parskip=\baselineskip
                    \parindent=0bp
                    #2
                }%
                \kern\columnsep%
            }%
            \kern\columnsep%
        }%
    }%
    \setbox\bkgdbox\vbox{
        \color{#1}
        \hrule width  \wd\contentbox %
               height \ht\contentbox %
               depth  \dp\contentbox
        \color{black}
    }%
    \wd\bkgdbox=0bp%
    \vbox{\hbox to \hsize{\box\bkgdbox\box\contentbox}}%
    \vskip\baselineskip%
}

\date{}

\begin{document}


\title{Solution semiflow to the isentropic Euler system}

\author{Dominic Breit}
\address[D. Breit]{Department of Mathematics, Heriot-Watt University, Riccarton Edinburgh EH14 4AS, UK}
\email{d.breit@hw.ac.uk}

\author{Eduard Feireisl}
\address[E.Feireisl]{Institute of Mathematics AS CR, \v{Z}itn\'a 25, 115 67 Praha 1, Czech Republic}
\email{feireisl@math.cas.cz}
\thanks{The research of E.F. leading to these results has received funding from
the Czech Sciences Foundation (GA\v CR), Grant Agreement
18--05974S. The Institute of Mathematics of the Academy of Sciences of
the Czech Republic is supported by RVO:67985840.}

\author{Martina Hofmanov\'a}
\address[M. Hofmanov\'a]{Fakult\"at f\"ur Mathematik, Universit\"at Bielefeld, D-33501 Bielefeld, Germany}
\email{hofmanova@math.uni-bielefeld.de}

\begin{abstract}

{It is nowadays well understood that the multidimensional isentropic Euler system is desperately ill--posed.}
{Even certain smooth initial data give rise to infinitely many solutions and all available  selection criteria fail to  ensure both global existence and uniqueness.} We propose a different approach to well--posedness of this system based on ideas from the theory of Markov semigroups: we show the existence of a Borel measurable solution semiflow. To this end, we introduce a notion of dissipative solution which is understood as time dependent trajectories of the basic state variables -
the mass density, the linear momentum, and the energy - in a suitable phase space. The underlying system of PDEs is satisfied in a generalized sense. The solution semiflow enjoys the standard semigroup property and the solutions coincide with the strong solutions
as long as the latter exist. Moreover, they minimize the energy (maximize the energy dissipation) among all dissipative solutions.

\end{abstract}


\date{\today}

\maketitle

\tableofcontents

\section{Introduction}
\label{I}

The motion of a compressible isentropic fluid in the Eulerian reference frame is described by the time evolution of
the mass density $\vr = \vr(t,x)$, $t \geq 0$, $x \in Q \subset R^N$, $N=1,2,3,$ and the momentum $\vm = \vm(t,x)$ solving the
\emph{Euler system}:
\begin{equation} \label{I1}
\begin{split}
\partial_t \vr + \Div \vm = 0,\\
\partial_t \vm + \Div \left( \frac{\vm \otimes \vm}{\vr} \right) + a \Grad \vr^\gamma = 0,\ { a > 0},
\end{split}
\end{equation}
where $\gamma > 1$ is the adiabatic constant. The problem is closed by prescribing the \emph{initial data}
\begin{equation} \label{I2}
\vr(0, \cdot) = \vr_0, \ \vm(0,\cdot) = \vm_0,
\end{equation}
as well as appropriate boundary conditions. For the sake of simplicity, we eliminate possible problems connected with the presence of kinematic boundary by considering the
space--periodic flows, for which the physical domain can be identified with the flat torus,
\begin{equation} \label{I5}
Q = T^N = \left\{ [-1,1]|_{\{-1; 1\}} \right\}^N.
\end{equation}

It is well--known that solutions of \eqref{I1} develop singularities -- shock waves -- in finite time no matter how smooth or
small the initial data are. Accordingly, the concept of weak (distributional) solution has been introduced to study global--in--time behavior of system \eqref{I1}. The existence of weak solutions in the simplified monodimensional geometry has been established
for a rather general class of initial data, see Chen and Perepelitsa \cite{ChenPer2}, DiPerna \cite{DiP3}, Lions, Perthame and Souganidis \cite{LPS}, among others. More recently, the theory of convex integration has been used to show existence of weak solutions for $N=2,3$ again for a rather vast class of data, see Chiodaroli \cite{Chiod}, De Lellis and Sz\' ekelyhidi \cite{DelSze3},
Luo, Xie and Xin \cite{LuXiXi}.

Uniqueness and stability with respect to the initial data in the framework of weak solutions is a more delicate issue. Apparently, the
Euler system is ill--posed in the class of weak solutions and explicit
examples of multiple solutions emanating from the same initial state have been constructed, see e.g. the monograph of Smoller \cite{SMO}.
An admissibility criterion must be added to the weak formulation of \eqref{I1} in order  to select the physically relevant solutions. To this
end, consider the total energy $e$ given by
\[
e(\vr, \vm) = e_{\rm kin}(\vr, \vm) + e_{\rm int}(\vr),\ e_{\rm kin} (\vr,\vm)= \frac{1}{2} \frac{|\vm|^2}{\vr},\ e_{\rm int} (\vr)= \frac{a}{\gamma - 1} \vr^\gamma.
\]
The \emph{admissible solutions} satisfy, in addition to the weak version of \eqref{I1}, the total energy balance
\begin{equation} \label{I3}
\partial_t  e(\vr, \vm) + \Div \left[ \left(e(\vr, \vm) + a \vr^\gamma \right) \frac{\vm}{\vr} \right] \leq 0,
\end{equation}
or at least its integrated form,
\begin{equation} \label{I4}
\frac{{\rm d}}{{\rm d}t} E(t) \leq 0,\ E \equiv \int_{Q} \left[ \frac{|\vm|^2}{\vr} + \frac{a}{\gamma - 1} \vr^\gamma \right] \dx.
\end{equation}
Note that \eqref{I4} follows directly form \eqref{I3} thanks to the periodic boundary conditions; the same holds,
of course,
under suitable conservative boundary conditions, for instance,
\[
\vm \cdot \vc{n}|_{\partial Q} = 0.
\]
Note that the \emph{inequality} {in  \eqref{I3} is needed to select the physically relevant discontinuous shock--wave solutions.}

Even if \eqref{I3} is imposed as an extra selection criterion, the weak solutions are still not unique, see Chiodaroli, De Lellis and Kreml \cite{ChiDelKre}, Markfelder and Klingenberg \cite{MarKli}. The initial data giving rise to infinitely many admissible solutions are termed \emph{wild data}. As shown in \cite{ChiDelKre}, this class {includes certain Lipschitz initial data}. {Recently, 
this result has been extended to smooth initial data by Chiodaroli et al \cite{ChKrMaSw}.} Furthermore, even if additional selection criteria as, for instance, maximality of the energy dissipation,
are imposed, the problem remains ill--posed, see Chiodaroli and Kreml \cite{ChiKre}.

An important feature of systems with uniqueness is their  semiflow property: Letting the system run from time $0$ to time $s$ and then restarting and letting it run from time $s$ to time $t$ gives the same outcome as letting it run directly from time $0$ to time $t$. In other words, the knowledge of the whole past up to time $s$ provides no more useful information about the outcome at time $t$ than knowing the state of the system at time $s$ only. For systems where the uniqueness is unknown or not valid, a natural question is  whether a solution semiflow can be constructed anyway.

Therefore, inspired by the recent work of Cardona and Kapitanski \cite{CorKap}, we  propose a different approach to well--posedness of the Euler
system based on the theory of Markov selection in stochastic analysis, see e.g. Krylov \cite{KrylNV}, Stroock and Varadhan \cite{StrVar}, Flandoli and Romito \cite{FlaRom}, or \cite{BFH18markov}. More specifically, we establish the existence of a \emph{semiflow selection} for the Euler system
\eqref{I1}--\eqref{I5}, that is, a mapping
\[
{U}: [t, \vr_0, \vm_0, E_0] \mapsto  [\vr(t), \vm(t), E(t)], \ t \geq 0
\]
enjoying the semigroup property:
\begin{equation} \label{I6}
{U} [t_1 + t_2, \vr_0, \vm_0, E_0] = {U} \left[ t_2, {U} [t_1, \vr_0, \vm_0, E_0] \right]
\ \mbox{for any}\  t_1, t_2\geq 0,
\end{equation}
where $[\vr, \vm]$ represents a generalized solution to \eqref{I1}--\eqref{I5} with the energy $E$. More specifically,
the triple $[\vr, \vm, E]$ termed \emph{dissipative solution} will coincide with the expected value of
suitable measure--valued solution satisfying the Euler system \eqref{I1}, together with the energy inequality \eqref{I4}, satisfied
in a generalized sense. The precise definitions may be found in Section  \ref{W}.  In addition to the semigroup property \eqref{I6}, the semiflow we shall construct enjoys the following properties, which provides further justification of the physical relevance of our construction:
\begin{itemize}
\item{\bf Stability of strong solutions.}
Let the Euler system \eqref{I1}--\eqref{I5} admit a strong $C^1$ solution $\widehat{\vr}$, $\widehat{\vm}$, with the
associated energy
\[
E_0 = \int_{Q} \left[ \frac{|\vm_0|^2}{\vr_0} + \frac{a}{\gamma - 1} \vr_0^\gamma \right] \dx,
\]
defined on a maximal time interval $[0, T_{\rm max})$.

Then we have
\[
U[t, \vr_0, \vm_0, E_0] = [\widehat{\vr}, \widehat{\vm}, E_0](t) \ \mbox{for all}\ t \in [0,T_{\rm max}).
\]
{This reflects the fact that dissipative solutions satisfy the weak--strong uniqueness principle.} 

\item{\bf Maximal dissipation.} Let the Euler system \eqref{I1}--\eqref{I5} admit a dissipative solution  $\widehat{\vr}$, $\widehat{\vm}$, with the
associated energy $\widehat{E}$ such that
\[
\widehat{E}(t) \leq E(t) \ \mbox{for all}\ t \geq 0,
\]
where $E$ is the energy of the solution semiflow $U[t, \vr_0, \vm_0, E_0]$.

Then we have
\[
E(t) = \widehat{E}(t) \ \mbox{for all}\ t \geq 0.
\]
In other words, our search for physically relevant solutions respects the ideas of Dafermos \cite{Dafer} who introduced the selection criterion based on the maximization of the energy dissipation  for hyperbolic systems of conservation laws.

\item{\bf Stability of stationary states.} Let $\Ov{\vr} > 0$, $\vm \equiv 0$ be a stationary solution of the Euler system
\eqref{I1}--\eqref{I5}. Suppose that
\[
\vr(T, \cdot) = \Ov{\vr},\ \vm(T, \cdot) = 0 \ \mbox{for some}\ T \geq 0,
\]
where $\vr$, $\vm$ are the density and the momentum components of a solution semiflow $U[t, \vr_0, \vm_0, E_0]$.

Then we have
\[
\vr(t, \cdot) = \Ov{\vr},\ \vm(t, \cdot) = 0 \ \mbox{for all}\ t \geq T.
\]
Hence, if the system reaches a stationary state where the density is constant and the momentum vanishes, it remains in this state for all future times.

\end{itemize}

The fact that certain form of an energy inequality has to be included as an integral part of the definition of solution is pertinent to the analysis of problems in fluid mechanics. One of the main novelties of our approach is  including the total energy $E$ as a third variable in the construction of the semiflow. Intuitively speaking, the knowledge of the initial state for the density and the momentum does not provide sufficient information to restart the semiflow. We have already observed a similar phenomenon in the context of Markov selection for stochastic compressible Navier--Stokes system in \cite{BFH18markov}.

Our definition of dissipative solution is motivated by the notion of dissipative measure--valued solution known e.g. from \cite{FGSWW1}, \cite{GSWW}. However, we chose a different formulation which in our opinion reflects better the nature of the system and is more suitable for the construction of the solution semiflow. Similarly to the notion of dissipative measure--valued solution, our definition   permits to establish the weak--strong uniqueness principle. Consequently, strong solutions are \emph{always} contained in the selected semiflow as long as they exist.

Note that this desirable
property is not granted for the semiflow of weak solutions to the incompressible Navier--Stokes system presented in \cite{CorKap}. 
More precisely, even for the incompressible Navier--Stokes system, where global existence of unique solutions has not yet been excluded for smooth initial data, the  semiflow in \cite{CorKap}
may ``select'' completely pathological solutions like those that start from zero but have positive energy at later times (such solutions may exist thanks to the recent work by Buckmaster and Vicol \cite{BV18}).

To conclude this introduction, we remark that our method applies mutatis mutandis to the incompressible Navier--Stokes and Euler system as well as to the isentropic Navier--Stokes system. We have chosen the isentropic Euler system for this paper as it is the system where uniqueness seems to be the most out of reach. However, it would be interesting to investigate whether for one of the ``easier'' systems one could understand further  properties of the solution semiflow such as the dependence on the initial data. Moreover, uniqueness of the solution semiflow is also an open problem.

The paper is organized as follows. In Section \ref{W}, we introduce the concept of dissipative solution and state the main result concerning the semiflow selection. Section \ref{ES} is devoted to the proof of existence and stability of the dissipative solutions.
In Section \ref{A}, we present the abstract setting and in Section~\ref{SSS}, we show the existence of the semiflow selection. Section \ref{C} contains concluding discussion concerning refined properties of the constructed semiflow.

\section{Set--up and main results}
\label{W}
In this section we present several definitions of generalized solutions to the compressible Euler system. In particular, we introduce the dissipative solutions
and explain the concept of admissibility.
Finally, we present our main result on semiflow selection in
Section \ref{S}.

\subsection{Weak solutions}

Weak solutions of the Euler system \eqref{I1} on the time interval $[0,\infty)$ satisfy the integral identities
\begin{equation} \label{W1}
\left[ \intQ{ \vr \varphi } \right]_{t = 0}^{t = \tau}
= \int_{0}^{\tau} \intQ{ \Big[ \vr \partial_t \varphi + \vm \cdot \Grad \varphi \Big] } \dt
\end{equation}
for any $\varphi \in C^1_c([0, \infty) \times {T^N})$, and
\begin{equation} \label{W2}
\left[ \intQ{ \vm \cdot \bfphi } \right]_{t = 0}^{t = \tau}
= \int_{0}^{\tau} \intQ{ \left[ \vm \cdot \partial_t \bfphi + \frac{\vm \otimes \vm}{\vr} : \Grad \bfphi
+ a \vr^\gamma \Div \bfphi \right] } \dt
\end{equation}
for any $\bfphi \in C^1_c([0, \infty) \times {T^N}; R^N)$. A weak solution
to \eqref{I1} is a pair of measurable functions $[\vr,\vc{m}]$ such that all integrals in \eqref{W1} and \eqref{W2} are well--defined.
In accordance with the energy inequality \eqref{I4}, we suppose additionally that
\begin{align*}
&E = E(t) \ \mbox{is a non--increasing function of} \ t,\\
&\intQ{ \left[ \frac{1}{2} \frac{|\vm|^2 }{\vr} + \frac{a}{\gamma - 1} \vr^\gamma\right] (\tau, \cdot)  } =
E(\tau) \ \mbox{for a.a.}\ \tau.
\end{align*}
Consistently with \eqref{W1}, \eqref{W2}, this can be put in a variational form,
\begin{equation} \label{EW4}
\left[  E \psi  \right]_{t = \tau_1 - }^{t = \tau_2 +} -
\int_{\tau_1}^{\tau_2} E(t) \partial_t \psi(t) \ \dt \leq 0,\ 0 \leq \tau_1 \leq \tau_2,\ E(0-) = E_0,
\end{equation}
for any $\psi \in C^1_c ([0, \infty))$, $\psi \geq 0$.

\subsection{Dissipative solutions}
If $N=2,3$, the energy inequality \eqref{EW4} seems to be the only source of {\it a priori} bounds. However, as indicated by the numerous
examples of ``oscillatory'' solutions (cf. \cite{Chiod}, \cite{DelSze3}) the set of all admissible weak solutions emanating from given initial data is not
closed with respect to the weak topology on the trajectory space associated with the energy bounds \eqref{EW4}. There are two potential sources
of difficulties:

\begin{itemize}

\item non--controllable \emph{oscillations} due to accumulation of singularities;

\item blow--up type collapse due to possible \emph{concentration} points.

\end{itemize}

\noindent
To accommodate the above mentioned singularities in the closure of the set of weak solutions, two kinds of tools are used:
(i) the Young measures describing the oscillations, (ii) concentration defect measures for concentrations,  see e.g.
Brenier, De Lellis, Sz\' ekelyhidi \cite{BrDeSz}.

Let
\[
\mathcal{S} = \left\{ [\tvr, \tvm] \ \Big| \ \tvr \geq 0, \ \tvm \in R^N \right\}
\]
be the phase space associated to the Euler system. Let $\mathcal{P}(\mathcal{S})$ denote the set of probability measures on $\mathcal{S}$ and let $\mathcal{M}^{+}(T^{N})$ and $\mathcal{M}^{+}(T^{N}\times S^{N-1})$, respectively, denote the set of positive bounded Radon measures on $T^{N}$ and $T^{N}\times S^{N-1}$, respectively, where $S^{N-1}\subset R^{N}$ denotes the unit sphere.
A
\emph{dissipative solution} is defined via the following quantities:

\begin{itemize}
\item the
Young measure:
\begin{equation} \label{D1}
(t,x) \mapsto \nu_x(t) \in L^\infty_{{\rm weak - (*)}}((0, \infty) \times \Q; \mathcal{P}(\mathcal{S}));
\end{equation}
\item the kinetic and internal energy concentration defect measures:
\begin{align} \label{D2a}
t &\mapsto \mathfrak{C}_{\rm kin}(t) \in L^\infty_{\rm weak - (*)}(0, \infty; \mathcal{M}^+(\Q)), \\
t &\mapsto \mathfrak{C}_{\rm int}(t) \in L^\infty_{\rm weak - (*)}(0, \infty; \mathcal{M}^+(\Q));\label{D2b}
\end{align}
\item the convective and pressure concentration defect measures:
\begin{align} \label{D3a}
t &\mapsto \mathfrak{C}_{\rm conv} (t) \in L^\infty_{\rm weak - (*)}
{\left(0,\infty; \mathcal{M}^+\left(\Q \times S^{N-1} \right) \right)},
\\ t &\mapsto \mathfrak{C}_{\rm press} (t) \in {L^\infty_{\rm weak - (*)}\left(0,\infty; \mathcal{M}^+(\Q) \right).
}
\label{D3b}
\end{align}

\end{itemize}

\noindent The constitutive relations
\[
\frac{\vm \otimes \vm}{\vr} = 2 \left( \frac{\vm}{|\vm|} \otimes \frac{\vm}{|\vm|} \right) \left[ \frac{1}{2} \frac{|\vm|^2}{\vr} \right],\ a \varrho^\gamma = (\gamma - 1) \left[ \frac{a}{\gamma - 1}  \vr^\gamma \right]
\]
enforce natural compatibility conditions
\begin{equation} \label{D5}
\mathfrak{C}_{\rm conv}(t,\dif x,\dif\xi) = {2 r_x(t,\dif\xi)  }\otimes \mathfrak{C}_{{\rm kin}}(t,\dif x),\quad
\mathfrak{C}_{\rm press} = (\gamma - 1) \mathfrak{C}_{\rm int},
\end{equation}
where $r_x(t)\in \mathcal{P} (S^{N-1})$ are the measures associated to disintegration of
$\mathfrak{C}_{\rm conv}(t)$ on the product $T^N \times S^{N-1}$, see e.g. Ambrosio, Fusco, and Palara \cite[Theorem 2.28]{AmFuPa}.

Hereafter, we denote
by $[\tilde\vr,\tilde\vm]$ the dummy variables in phase space $\mathcal{S}$ whereas $\xi$ is a dummy variable in $S^{N-1}$.
We are now in the position to present the basic building block for the semiflow selection -
a dissipative solution of the Euler system
\eqref{I1}. 

\begin{Definition}[Dissipative solution] \label{DD1}

The triple of functions
\[
{
[\vr, \vm, E]\in C_{\mathrm{weak,loc}}([0,\infty); L^\gamma (T^N)) \times C_{\mathrm{weak,loc}}([0,\infty); L^{\frac{2\gamma}{\gamma + 1}}(T^N; R^N))\times BV_{\mathrm{loc}}([0,\infty))}
\]
is called \emph{dissipative solution} of the Euler system
\eqref{I1} with the initial data
\[
{[\vr_0, \vm_0, E_0]\in L^\gamma(T^N)\times L^{\frac{2 \gamma}{\gamma + 1}}(T^N;R^N)\times [0,\infty) }
\]
if there exists a family of parametrized measures specified through
\eqref{D1}--\eqref{D5} such that:
\begin{itemize}
\item[a)] for a.a $\tau >0$ we have
\begin{equation} \label{D4}
\begin{split}
\vr(\tau,x) &= \left< \nu_{x}(\tau); \tvr \right> \geq 0, \
\vm(\tau,x) = \left< \nu_{x}(\tau); \tvm \right> \ \mbox{for a.a}\ x \in \Q, \\
E(\tau) &= \intQ{ \left< \nu_{x}(\tau); \frac{1}{2} \frac{|\tvm|^2}{\tvr} + \frac{a}{\gamma -1} \tvr^\gamma \right>}
+ \int_{\Q} \ \D \mathfrak{C}_{\rm kin}(\tau) + \int_{\Q} \ \D \mathfrak{C}_{\rm int}(\tau)\ ;
\end{split}
\end{equation}

\item[b)] for any $\tau > 0$ the integral identity
\begin{equation} \label{W4}
\left[ \intQ{ \vr \varphi } \right]_{t = 0}^{t = \tau}
= \int_{0}^{\tau} \intQ{ \Big[ \vr \partial_t \varphi + \vm \cdot \Grad \varphi \Big] } \dt,
\end{equation}
holds for
any {$\varphi \in C^1_c([0, \infty) \times {T^N})$}, where $\vr(0, \cdot) = \vr_0$;

\item[c)]
for any $\tau > 0$ the integral identity
\begin{equation} \label{W5}
\begin{split}
&\left[ \int_{\Q} \vm \cdot \bfphi \dx \right]_{t = 0}^{t = \tau}-  \int_{0}^{\tau} \intQ{\vm \cdot \partial_t \bfphi }\dt\\
&= \int_{0}^{\tau} \intQ{ \left[  \left< \nu_{x}(t); 1_{\tvr > 0}\frac{\tvm \otimes \tvm}{\tvr} \right> : \Grad \bfphi
+ \left< \nu_{x}(t); a \tvr^\gamma \right> \Div \bfphi \right] } \dt \\
&+ {
2 \int_{0}^{\tau} \int_{\Q}  \left< r_{x}(t); \xi \otimes \xi \right> : \Grad \bfphi \
 \ {\rm d}
\mathfrak{C}_{{\rm kin}} \,\dif t }
+ (\gamma - 1)\int_0^\tau \int_{\Q}  \Div \bfphi \ \D \mathfrak{C}_{{\rm int}}  \,\dif t,
\end{split}
\end{equation}
holds for
any $\bfphi \in C^1_c([0, \infty) \times {T^N}; R^N)$, where $\vm(0, \cdot) = \vm_0$;
\item[d)] for any $0 \leq \tau_1 \leq \tau_2$
the inequality
\begin{equation} \label{W6}
\begin{split}
\Big[ E \psi \Big]_{t = \tau_1-}^{t = \tau_2+} -
\int_{\tau_1}^{\tau_2} E \partial_t \psi \ \dt \leq 0 ,
\end{split}
\end{equation}
holds for any $\psi \in C^1_c([0, \infty))$ with $\psi \geq 0$, where $E(0-) = E_0$.
\end{itemize}

\end{Definition}

Note that our definition is slightly different from the one used by Gwiazda, \'Swierczewska-Gwiazda, and
Wiedemann \cite{GSWW} based on the concentration defect
measures introduced by Alibert and Bouchitt\' e \cite{AliBou}. We believe that the present setting based on the energy defects rather than the recession functions reflects better the underlying system of PDEs. It is also worth noting that the present definition contains definitely more
information on the dissipative solutions than its counterpart introduced in \cite{FGSWW1}
in the context of the compressible Navier--Stokes system. The class of solutions
considered in \cite{FGSWW1} is apparently larger but still guarantees the weak--strong uniqueness principle.
Indeed, the corresponding proof in \cite{FGSWW1} adapts easily to the Euler setting. In particular, we obtain
the following result that can be proved exactly as \cite{FGSWW1}, see also Gwiazda et al. \cite{GSWW} and Section \ref{REI} below.

\begin{Proposition}[Weak--strong uniqueness] \label{WP1}
Let $[\vr, \vm, E]$ be a dissipative solution to \eqref{I1} in the sense of Definition \ref{DD1} starting from the initial state
$[\vr_0,\ \vm_0, \ E_0]$, $\vr_0 > 0$.
Let $[\widehat{\vr}, \widehat{\vc{m}}]$ be a strong solution\footnote{A strong solution belongs to the class $W^{1,\infty}$ and satisfies \eqref{I1} a.e. pointwise.} to \eqref{I1}
in $[0, T_{\rm max}) \times T^N$ starting from the same initial data
$\widehat{\vr}_0 = \vr_0$, $\widehat{\vc{m}}_0 = \vm_0$, with
\[
\intQ{\left[ \frac{1}{2} \frac{|\widehat{\vc{m}}_0|^2}{\widehat{\vr}_0} + \frac{a}{\gamma - 1} {\widehat{\vr}}^\gamma_0 \right]} = E_0.
\]

Then we have
\[
\vr = \widehat{\vr}, \ \vm = \widehat{\vc{m}},\ E = E_0 \ \mbox{in}\ [0, T_{\rm max}) \times T^N.
\]

\end{Proposition}

\subsection{Admissible dissipative solutions}

Finally, we introduce a subclass of dissipative solutions that reflect the physical principle of maximization of the energy dissipation. To this end, let $[\vr^i, \vm^i, E^i]$, $i=1,2$, be two dissipative solutions starting from the same initial data $[\vr_0, \vm_0, E_0]$. We introduce the relation
\begin{equation*} 
[\vr^1, \vm^1, E^1] \prec [\vr^2, \vm^2, E^2]\ \Leftrightarrow \ E^{1}(\tau \pm) \leq E^{2}(\tau \pm)
\ \mbox{for any}\ \tau \in (0, \infty).
\end{equation*}

\begin{Definition}[Admissible dissipative solution] \label{DD2}

We say that a dissipative solution $[\vr, \vm, E]$ starting from the initial data $[\vr_0, \vm_0, E_0]$ is \emph{admissible} if
it is minimal with respect to the relation $\prec$. Specifically, if
\[
[\tvr, \tvm, \tilde{E}] \prec [\vr, \vm, E],
\]
where $[\tvr, \tvm, \tilde{E}]$ is another dissipative solution starting from $[\vr_0, \vm_0, E_0]$, then
\[
E = \tilde E  \ \mbox{in}\ [0, \infty).
\]
\end{Definition}
{Maximizing the energy dissipation or, equivalently, minimizing the total energy of the system is motivated by a similar selection criterion proposed by Dafermos \cite{Daf4}. In view of the arguments discussed in Section \ref{C}, such a selection criterion:
\begin{itemize}
\item rules out a large part of wild solutions obtained via ``available'' methods;
\item guarantees stability of equilibrium states in the class of dissipative solutions.
\end{itemize}

\subsection{Semiflow selection -- main result}
\label{S}

We start by introducing suitable topologies on the space of the initial data and the space of dissipative solutions.
Fix  $\ell>N/2+1$ and consider the Hilbert space
\[
X = W^{-\ell,2}(\Q) \times W^{-\ell,2}(\Q; R^N) \times R,
\]
together with its subset containing the initial data
\begin{equation*}
D = \left\{ [\vr_0, \vm_0, E_0]\in X\ \Big| \ \vr_0 \geq 0,\
\intQ{ \left[ \frac{1}{2} \frac{|\vm_0|^2}{\vr_0} + \frac{a}{\gamma - 1} \vr_0^\gamma \right]} \leq E_0 \right\}.
\end{equation*}
Here the convex function $[\vr, \vm] \mapsto \frac{|\vm|^2}{\vr}$ is defined for $\vr \geq 0$, $\vm \in R^N$ as
\[
\frac{|\vm|^2}{\vr} = \left\{ \begin{array}{l} 0 \ \mbox{if} \ \vm = 0, \\
\frac{|\vm|^2}{\vr} \ \mbox{if} \ \vr > 0,\\
\infty \ \mbox{otherwise.} \end{array} \right.
\]
Note that  $D$ is a closed convex subset of $X$.
We consider the trajectory space
\[
\Omega = C_{{\rm loc}} ([0, \infty); W^{-\ell,2}(\Q)) \times
C_{{\rm loc}} ([0, \infty); W^{-\ell,2}(\Q; R^N)) \times L^1_{\rm loc}(0, \infty),
\]
which is a separable metric space. Dissipative solutions $[\vr, \vm, E]$, as defined in Definition~\ref{DD1}, belong to this class. Indeed,  equations \eqref{W4} and \eqref{W5} give an information on the time regularity of the density and the momentum whereas the energy can be controlled by \eqref{W6}. Moreover, for initial data $[\vr_0, \vm_0, E_0] \in D$ it follows from \eqref{D4} and Jensen's inequality that a dissipative solution $[\vr,\vm,E]$ evaluated at a.a.  times $t\geq0$ also belongs to the set $D$.
Finally, for initial data $[\vr_0, \vm_0, E_0] \in D$, we introduce the solution set
\begin{equation*} 
\begin{split}
\mathcal{U} &[\vr_0, \vm_0, E_0] =
\left\{ [\vr, \vm, E] \in \Omega \ \Big| \begin{array}{c}
[\vr, \vm, E] \ \mbox{is a dissipative solution}\\ \mbox{with initial data}\ [\vr_0, \vm_0, E_0]\end{array} \right\}.
\end{split}
\end{equation*}
We are now ready to define a semiflow selection to \eqref{I1}.
\begin{Definition}[Semiflow selection] \label{SD1}

A \emph{semiflow selection} in the class of dissipative solutions for the compressible Euler system \eqref{I1} is a mapping
\[
U: D \to \Omega,\
U\left\{ \vr_0, \vm_0, E_0 \right\} \in \mathcal{U} [\vr_0, \vm_0, E_0] \ \mbox{for any}\
[\vr_0, \vm_0, E_0] \in D
\]
enjoying the following properties:
\begin{itemize}
\item[a)]
{\bf Measurability.} The mapping $U: D \to \Omega$ is Borel measurable.
\item[b)]
{\bf Semigroup property.} We have
\[
U \left\{ \vr_0, \vm_0, E_0 \right\} (t_1 + t_2) =
U \left\{ \vr(t_1), \vm(t_1), E(t_1-) \right\}(t_2)
\]
for any $[\vr_0, \vm_0, E_0] \in D$ and any $t_1, t_2\geq 0$, where
$[\vr, \vm, E] = U \left\{ \vr_0, \vm_0, E_0 \right\}$.

\end{itemize}

\end{Definition}

Our main result reads as follows.

\begin{Theorem} \label{ST1}

The isentropic Euler system \eqref{I1} admits a semiflow selection
$U$ in the class of dissipative solutions in the sense of Definition \ref{SD1}. Moreover, we have that
\[
U\left\{ \vr_0, \vm_0, E_0 \right\} \ \mbox{is admissible in the sense of Definition \ref{DD2}}
\]
for any $[\vr_0, \vm_0, E_0] \in D$.

\end{Theorem}
In the next section we prove the existence of at least one dissipative solution for given initial data and the sequential stability of the solution set.
The abstract setting for the selection principle is presented in Section~\ref{A} and the proof of Theorem~\ref{ST1} can be found in Section~\ref{SSS}.
The additional regularity properties of the selection mentioned in Section \ref{I} will be discussed in Section~\ref{C}.

\section{Existence and sequential stability}
\label{ES}

We aim to show:
\begin{itemize}
\item \emph{existence} of a dissipative solution for any initial data $[\vr_0, \vm_0, E_0] \in D$, meaning
\[
\mathcal{U}[\vr_0, \vm_0, E_0] \ne \emptyset;
\]
\item \emph{sequential stability} of the solution set, meaning
\[
\mathcal{U}[\vr_0, \vm_0, E_0] \subset \Omega \ \mbox{is compact}
\]
and the multivalued mapping
\[
[\vr_0, \vm_0, E_0] \in D \subset X \to \mathcal{U}[\vr_0, \vm_0, E_0] \in 2^\Omega
\]
has closed graph; whence by Lemma 12.1.8 in \cite{StrVar} it is (strongly) Borel measurable.

\end{itemize}

We note that if $\mathcal{U}[\vr_0, \vm_0, E_0]$ is a compact subset of the separable metric space $\Omega$
for any $[\vr_0, \vm_0, E_0] \in D$, then the (Borel) measurability of the multivalued mapping
\[
\mathcal{U}: D \to 2^\Omega
\]
corresponds to measurability with respect to the Hausdorff metric on the space of all compact subsets of $\Omega$.

\subsection{Sequential stability}

We first address the issue of sequential stability as the existence proof leans basically on identical arguments.

\begin{Proposition} \label{ESP2}

Suppose that $\{ \vr_{0,\ep}, \vm_{0,\ep}, E_{0,\ep} \}_{\ep > 0} \subset D$ is a sequence of data giving rise to a family of dissipative solutions
$\{ \vre, \vme, E_\ep \}_{\ep > 0}$, that is,
$
[\vre, \vme, E_\ep] \in \mathcal{U}[ \vr_{0,\ep}, \vm_{0,\ep},E_{0,\ep} ].
$
Moreover, we assume that there  exists $\Ov{E}>0$ such that $  E_{0,\ep} \leq \Ov{E}$ for all $\ep > 0$.

Then, at least for suitable subsequences,
\begin{equation} \label{ES1}
\vr_{0,\ep} \to \vr_0 \ \mbox{weakly in}\ L^\gamma(T^N), \
\vm_{0,\ep} \to \vm_0 \ \mbox{weakly in}\ L^{\frac{2 \gamma}{\gamma + 1}}(T^N; R^N)),\
E_{0,\ep} \to E_0.
\end{equation}
and
\[
\begin{split}
\vre &\to \vr \ \mbox{in}\ C_{{\rm weak, loc}}([0, \infty); L^\gamma (T^N)),\\
\vme &\to \vm \ \mbox{in}\ C_{{\rm weak, loc}}([0, \infty); L^\gamma (T^N ; R^N)),\\
E_\ep(\tau) &\to E(\tau) \ \mbox{for any}\ \tau \in [0, \infty) \ \mbox{and in}\ L^{1}_{\rm loc}(0,\infty),
\end{split}
\]
where
\[
[\vr, \vm, E] \in \mathcal{U} [\vr_0, \vm_0, E_0].
\]

\end{Proposition}

\begin{proof}

We proceed via several steps.

\begin{itemize}

\item

First of all, observe that the convergence \eqref{ES1} follows immediately from the fact that the energy $E_{0,\ep}$ is bounded
uniformly for $\ep \to 0$.

\item

It follows from Jensen's inequality that
\begin{equation}\label{eq:J1}
\frac{1}{2} \frac{|\vme|^2}{\vre} + \frac{a}{\gamma - 1} \vre^\gamma
\leq \left< \nu^\ep_x (t) ; \frac{1}{2} \frac{|\tvm|^2}{\tvr} + \frac{a}{\gamma - 1} \tvr^\gamma \right> \ \mbox{a.a. in}\
(0, \infty) \times \Q,
\end{equation}
where $\nu^{\ep}$ is the Young measure associated with the solution $[\vr_{\ep},\vm_{\ep},E_{\ep}]$.
Consequently, as $E_{0,\ep} \to E_0$, we deduce from the energy inequality \eqref{W6}, \eqref{D4}, {and the equations \eqref{W4}, \eqref{W5}}  that (up to a subsequence)
\[
\begin{split}
\vre &\to \vr \ \mbox{in}\ C_{{\rm weak, loc}}([0, \infty); L^\gamma(T^N)), \ \vr \geq 0,\\
\vme &\to \vm \ \mbox{in}\ C_{{\rm weak, loc}}([0, \infty); L^\frac{2\gamma}{\gamma+1}(T^N; R^N)),
\end{split}
\]
where
\[
\vr(0, \cdot) = \vr_0, \ \vm(0, \cdot) = \vm_0.
\]
In addition, note that, by \eqref{W6} and \eqref{D4}, the energy is non--increasing and non--negative; whence its total variation can be bounded by the initial value and the latter one is uniformly bounded by assumption. Hence by Helly's selection theorem, we have
\[
E_\ep(\tau) \to E(\tau) \ \mbox{for any}\ \tau \in [0, \infty) \ \mbox{and in}\ L^{1}_{\rm loc}(0,\infty),\ E(0+) \leq E_0.
\]

\item

In view of the above observations, it is easy to perform the limit in the equation of continuity \eqref{W4} to obtain
\[
\left[ \intQ{ \vr \varphi } \right]_{t = 0}^{t = \tau}
= \int_{0}^{\tau} \intQ{ \Big[ \vr \partial_t \varphi + \vm \cdot \Grad \varphi \Big] } \dt
\]
for any {$\varphi \in C^1_c([0, \infty) \times {T^N})$}, as well as in the energy balance \eqref{W6}: we get
\[
\begin{split}
\int_0^\infty E(t) \partial_t \psi \ \dt \geq 0 \ \mbox{for any}\ \psi \in C^1_c((0, \infty)),\ \psi \geq 0,
\ E(0+) \leq E_0,
\end{split}
\]
from which we deduce \eqref{W6}.
Moreover, we have
\[
\left[ \intQ{ \vme \cdot \bfphi } \right]_{t = 0}^{t = \tau} \to
\left[ \intQ{ \vm \cdot \bfphi } \right]_{t = 0}^{t = \tau},
\]
and
\[
\int_0^\tau \intQ{ \vme \cdot \partial_t \bfphi }\,\dif t \to
\int_0^\tau \intQ{ \vm \cdot \partial_t \bfphi }\,\dif t
\]
for any test function admissible in the momentum balance \eqref{W5}.

\item

Next, we denote by $\mathfrak{C}^{\ep}_{\rm kin}$ and $\mathfrak{C}^{\ep}_{\rm int}$ the kinetic and internal energy concentration defect measure associated with $[\vr_{\ep},\vm_{\ep},E_{
\ep}]$. Using again the energy inequality \eqref{W6} and \eqref{D4}, we deduce (up to a subsequence), the convergence of the concentration measures,
\[
\begin{split}
\mathfrak{C}^\ep_{\rm kin} &\to \mathfrak{C}^{\infty,1}_{\rm kin}\ \mbox{weakly-(*) in}\ L^\infty(0, \infty; \mathcal{M}^+(\Q)),\\
\mathfrak{C}^\ep_{\rm int} &\to \mathfrak{C}^{\infty,1}_{\rm int}\ \mbox{weakly-(*) in}\ L^\infty(0, \infty; \mathcal{M}^+(\Q)).
\end{split}
\]
In view of \eqref{D5} we denote by
\begin{equation}\label{D55}
\mathfrak{C}^{\ep,1}_{\rm conv}(t) \equiv 2r^\ep(t) \otimes \mathfrak{C}^{\ep}_{\rm kin}(t),
\end{equation}
the convective concentration defect measure associated with $[\vr_{\ep},\vm_{\ep},E_{\ep}]$ and  deduce
\begin{equation} \label{eq:1}
\mathfrak{C}^{\ep,1}_{\rm conv} \to \mathfrak{C}^{\infty,1}_{\rm conv} \ \mbox{weakly-(*) in}\ L^\infty(0, \infty;
\mathcal{M}^+(\Q \times S^{N-1})).
\end{equation}
We remark that the final convective concentration defect measure will be constructed below as a sum of $\mathfrak{C}^{\infty,1}_{\rm conv}$ and another measure obtained from the concentrations of the Young measures $\nu^{\varepsilon}_{x}(t)$.

With the above convergences at hand,  we are able to pass to the limit in the kinetic as well as internal energy concentration defect measure in \eqref{D4} and also in  the pressure concentration defect measure in \eqref{W5}.
Furthermore, we can pass to the limit in the integrals related to the convective term. More precisely, in view of \eqref{D55},
we have
\[
\begin{split}
2 &\int_{0}^{\infty} \int_{\Q}  \left< r^\ep_{x}(t); \xi \otimes \xi \right> : \Grad \bfphi  \ \mathfrak{C}^\ep_{{\rm kin}}(t,\dif x) \,\dif t
\\ &= \int_0^\infty \int_{\Q} \int_{S^{N-1}}(\xi \otimes \xi): \Grad \bfphi \  \mathfrak{C}_{\rm conv}^{\ep,1} (t,\dif x,\dif\xi)\,\dif t
\end{split}
\]
and the right hand side converges by \eqref{eq:1} to
$$
\int_0^\infty \int_{\Q} \int_{S^{N-1}} (\xi \otimes \xi): \Grad \bfphi \  \mathfrak{C}_{\rm conv}^{\infty,1} (t,\dif x,\dif \xi) \,\dif t.
$$
Finally, {we realize that} $2 \mathfrak{C}^{\infty,1}_{\mathrm{kin}}(t,\dif x)$ is the marginal of $\mathfrak{C}_{\mathrm{conv}}^{\infty,1}(t,\dif x,\dif\xi)$ corresponding to the variable $x$, that is,
\begin{equation} \label{3bis}
2 \mathfrak{C}^{\infty,1}_{\mathrm{kin}}(t,\dif x)=\mathfrak{C}_{\mathrm{conv}}^{\infty,1}(t,\dif x,S^{N-1}).
\end{equation}
Indeed, by \eqref{D55}, this is true on the approximate level and the property is preserved through the passage to the limit as $\varepsilon\to0$.

\item

Finally, it remains to handle the terms containing the Young measure. First, we
deduce from the energy inequality \eqref{W6} together with \eqref{D4} and \eqref{eq:J1} that the Young measures $\nu^{\varepsilon}_{x}(t)$ have uniformly bounded first moments. This  implies their (relative) compactness leading to
\begin{equation}\label{eq:2}
\nu^\ep_{x}(t)  \to \nu_{x}(t)
\ \mbox{weakly-(*) in}\ L^\infty \left((0, \infty ) \times \Q; \mathcal{P}(\mathcal{S}) \right).
\end{equation}
Note that the fact that the limit is again a (parametrized) probability measure follows from the finiteness of the first moments, see e.g.
Ball \cite{BALL2}.

Next,
let $\chi_k$ and and $\psi_k$, $k\in\mathbb{N}$, be cut--off functions satisfying
\begin{align*}
\chi_k &\in C^\infty (R),\
0 \leq \chi_k(Z) \leq 1,\\
\chi_k(Z) &= 0 \ \mbox{for}\ Z \leq k - 1,\ 0 \leq \chi_k(Z) \leq 1,\ \chi_k(Z) = 1\ \mbox{for}\ Z \geq k;\\
\psi_k &\in C^\infty(R), 0 \leq \psi_k(Z) \leq 1, \\
\psi_k(Z) &= 0\ \mbox{for}\ Z \leq 0,\ 0 \leq \psi_k(Z) \leq 1 \ \mbox{for} \ 
0 \leq Z \leq \frac{1}{k},\ \psi_k(Z) = 1 \ \mbox{for}\ Z \geq \frac{1}{k}. 
\end{align*}

We consider the two families of measures (here $b \in C(S^{N-1})$)
\[
\begin{split}
\mathfrak{C}^{\ep,k, 2}_{\rm conv} &\in L^\infty_{\rm weak-(*)}(0, \infty; \mathcal{M}^+ (T^N \times S^{N-1})),\\
\int_{S^{N-1}}b(\xi)\mathfrak{C}^{\ep,k, 2}_{\rm conv}(t,\dif x,{\dif \xi}) &= \left< \nu^\ep_x(t) ; 
b\left(\frac{\tilde\vm}{|\tilde\vm|}\right) \psi_k(\widetilde{\vr}) \widetilde\chi_k \left( \frac{1}{2} \frac{ |\tilde \vm |^2 }{\tilde \vr} \right) \frac{1}{2} \frac{ |\tilde \vm |^2 }{\tilde \vr}  \right>
\dx\\  
&+ \left< \nu_{t,x}; b\left(\frac{\tilde\vm}{|\tilde\vm|}\right)1_{\tvr > 0} (1 - \psi_k (\tvr)) \frac{1}{2} \frac{ |\tilde \vm |^2 }{\tilde \vr} \right>   ,
\end{split}
\]
and
\[
\mathfrak{C}^{\ep,k, 2}_{\rm press} \in L^\infty_{\rm weak-(*)}(0, \infty; \mathcal{M}^+ (T^N)),\
\mathfrak{C}^{\ep,k, 2}_{\rm press}(t,\dif x) = \left< \nu^\ep_x(t) ;
\chi_k \left(  a {\tilde \vr}^\gamma  \right)  a {\tilde \vr}^\gamma  \right> \dx.
\]
Due to \eqref{W6} and \eqref{D4} they are bounded uniformly in $\varepsilon,k$ and hence passing to the limit, first for $\ep \to 0$ then $k \to \infty$ we obtain
\begin{equation}\label{eq:J2}
\mathfrak{C}^{\ep,k, 2}_{\rm conv} \to \mathfrak{C}^{\infty, 2}_{\rm conv} \
\mbox{weakly-(*) in}\ L^\infty_{\rm weak-(*)}(0, \infty; \mathcal{M}^+ (T^N \times S^{N-1})),
\end{equation}
\[
\mathfrak{C}^{\ep,k, 2}_{\rm press} \to \mathfrak{C}^{\infty, 2}_{\rm press} \
\mbox{weakly-(*) in}\ L^\infty_{\rm weak-(*)}(0, \infty; \mathcal{M}^+ (T^N)).
\]
We set
\begin{equation} \label{4bis}
\mathfrak{C}^{\infty,2}_{\rm kin}(t) = \mathfrak{C}^{\infty,2}_{\rm conv}(t,\dx, S^{N-1}),\
\mathfrak{C}^{\infty,2}_{\rm int}(t) = \frac{1}{\gamma-1} \mathfrak{C}^{\infty,2}_{\rm press}.
\end{equation}

Accordingly, the convective term in the momentum equation \eqref{W5} can be decomposesd as
\[
\begin{split}
&\left< \nu^\ep_{x}(t);  1_{\tvr > 0} \frac{\tvm \otimes \tvm}{\tvr} \right> \dx \\
&=\left< \nu^\ep_{x}(t); \left( \frac{\tvm \otimes \tvm}{\tvr} \right)\psi_k(\tvr) (1 - \chi_k) \left( \frac{1}{2} \frac{|\tvm|^2}{\tvr} \right) \right> \dx\\
&+ \left< \nu^\ep_{x}(t); \left( \frac{\tvm \otimes \tvm}{\tvr} \right) 
\psi_k (\tvr)
\chi_k \left( \frac{1}{2} \frac{|\tvm|^2}{\tvr} \right) \right>
\dx \\
&+\left< \nu^\ep_{x}(t); \left( \frac{\tvm \otimes \tvm}{\tvr} \right)
 1_{\tvr > 0}(1 - \psi_k(\tvr)) \right> \dx\\
&=\left< \nu^\ep_{x}(t); \left( \frac{\tvm \otimes \tvm}{\tvr} \right) \psi_k (\tvr)(1 - \chi_k) \left( \frac{1}{2} \frac{|\tvm|^2}{\tvr} \right) \right>\dx\\
&+\left< \nu^\ep_{x}(t); \left( \frac{\tvm}{|\tvm|} \otimes \frac{\tvm}{|\tvm|} \right) 
\psi_k(\tvr)
\chi_k \left( \frac{1}{2} \frac{|\tvm|^2}{\tvr} \right) \frac{1}{2} \frac{|\tvm|^2}{\tvr}\right>\dx \\
& +\left< \nu^\ep_{x}(t); 
 \left( \frac{\tvm}{|\tvm|} \otimes \frac{\tvm}{|\tvm|} \right)   
1_{\tvr > 0}(1 - \psi_k(\tvr))\frac{1}{2} \frac{|\tvm|^2}{\tvr} \right> \dx
\end{split}
\]
Thus performing successively the limits $\ep \to 0$, $k \to \infty$ we obtain
\[
\left< \nu^\ep_{x}(t); \left( \frac{\tvm \otimes \tvm}{\tvr} \right) 
\psi_k(\tvr)
(1 - \chi_k) \left( \frac{1}{2} \frac{|\tvm|^2}{\tvr} \right) \right>
\dx \to \left< \nu_{x}(t); 1_{\tvr > 0}\left( \frac{\tvm \otimes \tvm}{\tvr} \right)  \right> \dx.
\]
Indeed, the passage to the limit as $\varepsilon\to0$ is follows from \eqref{eq:2} since the Young measures are applied to continuous and bounded functions, whereas the passage to the limit $k\to\infty$ is a consequence of dominated convergence together with the energy inequality \eqref{W6} and \eqref{D4}.\\
On the other hand, by definition of $\mathfrak{C}^{\varepsilon,k,2}_{\rm conv}$ and \eqref{eq:J2} we obtain
\begin{align*}
&\left< \nu^\ep_{x}(t); \left( \frac{\tvm}{|\tvm|} \otimes \frac{\tvm}{|\tvm|} \right) 
\psi_k (\tvr)
\chi_k \left( \frac{1}{2} \frac{|\tvm|^2}{\tvr} \right) \frac{1}{2} \frac{|\tvm|^2}{\tvr}\right>\dx\\
&+ \left< \nu^\ep_{x}(t); 
 \left( \frac{\tvm}{|\tvm|} \otimes \frac{\tvm}{|\tvm|} \right)   
1_{\tvr > 0}(1 - \psi_k(\tvr))\frac{1}{2} \frac{|\tvm|^2}{\tvr} \right> \dx \\
 &=\int_{S^{N-1}}(\xi\otimes\xi)\mathfrak{C}^{\varepsilon,k,2}_{\rm conv}(t,\dx, \dif\xi)\to
\int_{S^{N-1}}(\xi\otimes\xi)\mathfrak{C}^{\infty,2}_{\rm conv}(t,\dx, \dif\xi).
\end{align*}
The pressure term can be handled in a similar manner.\\
Finally, we set
\[
\mathfrak{C}_{\rm kin} = \mathfrak{C}_{\rm kin}^{\infty,1} + \mathfrak{C}^{\infty,2}_{\rm kin},\
\mathfrak{C}_{\rm int} = \mathfrak{C}_{\rm int}^{\infty,1} + \mathfrak{C}_{\rm int}^{\infty,2}
= \mathfrak{C}_{\rm int}^{\infty,1} + \frac{1}{\gamma - 1} \mathfrak{C}_{\rm press}^{\infty,2}
\]
and use relations \eqref{3bis}, \eqref{4bis} to obtain, after final disintegration
\[
\mathfrak{C}_{\rm conv} = \mathfrak{C}^{ \infty,1}_{\rm conv} +
\mathfrak{C}^{\infty,2}_{\rm conv} = 2 r_x(t) \otimes \left( \mathfrak{C}_{\rm kin}^{\infty,1} + \mathfrak{C}^{\infty,2}_{\rm kin} \right)
\]
for some measures $r_x(t) \in \mathcal{P} (S^{N-1})$.
\item Finally, we can pass to the limit in \eqref{D4}. Arguing as for the convective term we obtain
\begin{align*}
& \intQ{ \left< \nu^\varepsilon_{x}(\tau); \frac{1}{2} \frac{|\tvm|^2}{\tvr} + \frac{a}{\gamma -1} \tvr^\gamma \right>}\\
&\rightarrow \intQ{ \left< \nu_{x}(\tau); \frac{1}{2} \frac{|\tvm|^2}{\tvr} + \frac{a}{\gamma -1} \tvr^\gamma \right>}
+ \int_{\Q} \ \D {\mathfrak{C}_{\rm kin}^{\infty,2}}(\tau) + \int_{\Q} \ \D \mathfrak{C}_{\rm int}^{\infty,2}(\tau)
\end{align*}
weakly-(*) in $L^\infty(0,T)$ such that
$$E(\tau) = \intQ{ \left< \nu_{x}(\tau); \frac{1}{2} \frac{|\tvm|^2}{\tvr} + \frac{a}{\gamma -1} \tvr^\gamma \right>}
+ \int_{\Q} \ \D \mathfrak{C}_{\rm kin}(\tau) + \int_{\Q} \ \D \mathfrak{C}_{\rm int}(\tau).$$
\end{itemize}
The proof is hereby complete.
\end{proof}

\subsection{Existence}
The sequential stability from the previous part combined with a suitable approximation implies the existence of a dissipative solution. The precise statement is the content of the following proposition.
\begin{Proposition} \label{ESP1}

Let $[\vr_0,\vc{m}_0,E_0]\in D$ be given.
Then the isentropic Euler system \eqref{I1} admits a dissipative solution in the sense of Definition
\ref{DD1} with the initial data
$[\vr_0, \vm_0, E_0]$.

\end{Proposition}

\begin{proof} We adapt the method of Kr\" oner and Zajaczkowski \cite{KrZa} adding an artificial viscosity term of higher order to
the momentum equation. First observe that the definition of $D$ implies
\[
\vr_0 \in L^\gamma(\Q), \ \vm_0 \in L^{\frac{2\gamma}{\gamma + 1}}(\Q; R^N)
\]
with the respective bounds in terms of $E_{0}$.
It is a routine matter to construct approximating sequences satisfying,
\[
\vr_{0,\ep} \to \vr_0 \ \mbox{in}\ L^\gamma(\Q),\
\vm_{0,\ep} = \vr_{0,\ep} \vu_{0,\ep} \to \vm_0 \ \mbox{in}\ L^{\frac{2\gamma}{\gamma + 1}}(\Q; R^N)
\]
and
\[
\intQ{ \left[ \frac{1}{2}\vr_{0,\ep} |\vu_{0,\ep}|^2 + \frac{a}{\gamma - 1} \vr_{0,\ep}^\gamma \right] } \to
\intQ{\left[  \frac{1}{2} \frac{ |\vm_0|^2 }{\vr_0} + \frac{a}{\gamma - 1} \vr^\gamma_0 \right]}
\]
as $\ep \to 0$, where $\vr_{0,\ep} > 0$ and the velocity $\vu_{0,\ep}$
are smooth functions.

We consider the ``multipolar fluid'' type approximation of the Euler system \eqref{I1}:
\begin{equation} \label{ES11}
\begin{split}
\partial_t \vr + \Div (\vr \vu) &= 0,\\
\partial_t (\vr \vu) + \Div (\vr \vu \otimes \vu) + a \Grad \vr^\gamma &= -\ep \Del^{2m} \vu,
\end{split}
\end{equation}
where $\varepsilon>0$, $m\in\mathbb{N}$, and  the initial data is chosen as
\begin{equation} \label{ES12}
\vr(0, \cdot) = \vr_{0,\ep},\ \vu(0, \cdot) = \vu_{0,\ep}.
\end{equation}
It is well known that for $m\in \mathbb{N}$ large enough, see e.g. \cite{KrZa},  the problem \eqref{ES11}, \eqref{ES12} admits a unique smooth solution $[\vre, \vue]$ on the time interval $(0, \infty)$. Moreover, we have the total energy balance,
\begin{equation} \label{ES13}
\frac{{\rm d}}{\dt} \intQ{ \left[ \frac{1}{2}\vr_{\ep} |\vu_{\ep}|^2 + \frac{a}{\gamma - 1} \vr_{\ep}^\gamma \right] }
+ \ep \intQ{ |\Del^m \vu_\ep|^2 } = 0.
\end{equation}

Using the arguments of the preceding section, it is easy to perform the limit $\ep \to 0$ in the sequence of approximate solutions
\[
\left\{ \vre, \vm_\ep = \vre \vue, E_\ep = \intQ{ \left[ \frac{1}{2}\vr_{\ep} |\vu_{\ep}|^2 + \frac{a}{\gamma - 1} \vr_{\ep}^\gamma \right] }\right\}_{\ep > 0}
\]
to obtain the desired dissipative solution as long as we control the artificial viscosity terms. However, this is standard as \eqref{ES13} yields
\[
\sqrt{\ep} \Del^m \vue \ \mbox{bounded in}\ L^2(0, \infty; L^2(T^N; R^N))
\ \mbox{uniformly for} \ \ep \to 0.
\]
Accordingly, the corresponding term in the weak formulation of the momentum equation $\eqref{ES11}_2$ can be handled as
\begin{align*}
\left| \ep \int_0^\tau \intQ{ \Del^{2m} \vue \cdot \bfphi } \dt \right|
&= \left| \ep \int_0^\tau \intQ{ \Del^{m} \vue \cdot \Del^m \bfphi } \dt \right|\\
&\aleq \sqrt{\ep} \sup_{t \in [0,\tau]} \| \Del^m \bfphi \|_{L^\infty(T^N; R^N)}
\end{align*}
and vanishes asymptotically.
\end{proof}

\section{Abstract setting}
\label{A}

Our goal is to adapt the abstract machinery developed by Cardona and Kapitanski \cite{CorKap} to the family
$\mathcal{U}[\vr_0, \vm_0, E_0]$, $[\vr_0, \vm_0, E_0] \in D$. The following statement is a direct consequence
of Propositions \ref{ESP2} and \ref{ESP1}.

\begin{Lemma} \label{AL1}

For any $[\vr_0, \vm_0, E_0] \in D$, the set $\mathcal{U}[\vr_0, \vm_0, E_0]$ is a non--empty, compact subset
of $\Omega$.
Moreover, $[\vr(T), \vm(T), \mathcal{E} ] \in D$ for any $T > 0$, and for arbitrary $\mathcal{E} \geq E(T+)$.

\end{Lemma}

\subsection{Shift and continuation operations}
Two main ingredients for the construction of the semiflow are the shift invariance property
and the continuation property of the set of solutions (this corresponds to the disintegration and reconstruction property in the probabilistic setting of Markov selections).
For $\omega \in \Omega$, we define the positive shift operator
\[
S_T \circ \omega,\ S_T \circ \omega(t) = \omega(T + t),\ t \geq 0.
\]

\begin{Lemma}[Shift invariance property] \label{AL2}

Let $[\vr_0, \vm_0, E_0] \in D$ and $[\vr, \vm, E] \in \mathcal{U}[\vr_0, \vm_0, E_0]$.
Then we have
\[
S_T \circ [\vr, \vm, E] \in \mathcal{U}[\vr (T), \vm(T), \mathcal{E}]
\]
for any $T > 0$, and any $\mathcal{E} \geq E(T+)$.

\end{Lemma}

\begin{proof}

Obviously, a dissipative solution on the time interval $(0, \infty)$ solves also the same problem on $(T, \infty)$ with the initial data
$[\vr(T, \cdot), \vm(T, \cdot), E(T+)]$. Moreover, the energy is non--increasing; whence
\[
\lim_{t \to T+} E(t) = E(T+) \leq \mathcal{E}.
\]
The rest follows by shifting the test functions in the integrals.
\end{proof}

For $\omega_1, \omega_2 \in \Omega$ we define the continuation operator $\omega_1 \cup_T \omega_2$ by
\[
\omega_1 \cup_T \omega_2 (\tau) = \left\{
\begin{array}{l}
\omega_1(\tau) \ \mbox{for}\ 0 \leq \tau \leq T,\\ \\
\omega_2(\tau - T) \ \mbox{for}\ \tau > T. \end{array} \right.
\]

\begin{Lemma}[Continuation property] \label{AL3}

Let $[\vr_0, \vm_0, E_0] \in D$ and
\[
[\vr^1, \vm^1, E^1] \in \mathcal{U} [\vr_0, \vm_0, E_0],
\ [\vr^2, \vm^2, E^2] \in \mathcal{U} [\vr^{1}(T), \vm^{1}(T), \mathcal{E}] \ \mbox{for some}\  \mathcal{E} \leq E^1(T-).
\]

Then
\[
[\vr^1, \vm^1, E^1] \cup_T [\vr^2, \vm^2, E^2] \in \mathcal{U} [\vr_0, \vm_0, E_0].
\]

\end{Lemma}

\begin{proof}

We have only to realize that the energy of the solution $[\vr^1, \vm^1, E^1] \cup_T [\vr^2, \vm^2, E^2]$ indeed remains non--increasing
on $(0, \infty)$.
\end{proof}

\subsection{General ansatz}

Summarizing the previous part of this section and the results of Section \ref{ES}, we have shown the existence of a set--valued mapping
\[
D\ni[\vr_0, \vm_0, E_0] \mapsto \mathcal{U}[\vr_0, \vm_0, E_0] \in 2^\Omega
\]
enjoying the following properties:

\medskip

\begin{enumerate}[label={\bf (A\arabic{*})}]

\item\label{A1} Compactness: For any $[\vr_0, \vm_0, E_0] \in D$, the set $\mathcal{U}[\vr_0, \vm_0, E_0]$ is a non--empty compact
subset of $\Omega$.

\item\label{A2} Measurability: The mapping
\[
D\ni[\vr_0, \vm_0, E_0] \mapsto \mathcal{U}[\vr_0, \vm_0, E_0] \in 2^\Omega
\]
is Borel measurable, where the range of $\mathcal{U}$ is endowed with the Hausdorff metric on the subspace of compact sets in $2^\Omega$.

\item\label{A3} Shift invariance: For any
\[
[\vr, \vm, E] \in \mathcal{U}[\vr_0, \vm_0, E_0],
\]
we have
\[
S_T \circ [\vr, \vm, E] \in \mathcal{U}[\vr(T), \vm(t), E(T-)]\ \mbox{for any}\ T > 0.
\]

\item\label{A4} Continuation: If $T > 0$, and
\[
[\vr^1, \vm^1, E^1] \in \mathcal{U} [\vr_0, \vm_0, E_0],
\ [\vr^2, \vm^2, E^2] \in \mathcal{U} [\vr^1(T), \vm^1(T), E^1(T-)],
\]
then
\[
[\vr^1, \vm^1, E^1] \cup_T [\vr^2, \vm^2, E^2] \in \mathcal{U} [\vr_0, \vm_0, E_0].
\]
\end{enumerate}

The conditions \ref{A1}--\ref{A4} have been introduced in Cardona and Kapitanski \cite{CorKap}. In what follows, we will adopt
their method based on the ideas of Krylov \cite{KrylNV} and Strook and Varadhan \cite{StrVar} to select the desired solution semiflow.
We remark that the value $E(T-)$ in \ref{A3} and \ref{A4} can be replaced by
\[
\mathcal{E} = \eta E(T-) + (1 - \eta) E(T+)
\]
where $\eta \in [0,1]$ is given.

\section{Semiflow selection}
\label{SSS}

Following the general method by Krylov \cite{KrylNV}, we consider the family of functionals
\[
I_{\lambda, F}[\vr, \vm, E] = \int_0^\infty \exp( - \lambda t) F(\vr(t), \vm(t), E(t)) \, \dt ,\ \lambda > 0,
\]
where
\[
F: X = W^{-\ell,2}(\Q) \times W^{-\ell,2}(\Q; R^N) \times R\to R
\]
is a bounded and continuous functional.
Given $I_{\lambda, F}$ and a set--valued mapping $\mathcal{U}$ we define a selection mapping
$
I_{\lambda,F} \circ \mathcal{U},
$ by
\begin{equation*} 
\begin{split}
&I_{\lambda, F} \circ \mathcal{U}[\vr_0, \vm_0, E_0]\\
&= \left\{  [\vr, \vm, E] \in \mathcal{U}[\vr_0, \vm_0, E_0] \ \Big|
\begin{array}{c}I_{\lambda, F} [\vr, \vm, E]  \leq I_{\lambda, F}[\tvr, \tvm, \tilde{E}] \\ \ \mbox{for all}\ [\tvr, \tvm, \tilde{E}] \in \mathcal{U}[\vr_0, \vm_0, E_0] \end{array}  \right\}.
\end{split}
\end{equation*}
In other words, the selection is choosing minima of the functional $I_{\lambda,F}$. Note that a minimum exists since  $I_{\lambda,F}$ is continuous  on $\Omega$ and the set $\mathcal{U}[\vr_{0},\vm_{0},E_{0}]$ is compact in $\Omega$.
We obtain the following result for the set $I_{\lambda, F} \circ \mathcal{U}$.

\begin{Proposition} \label{AP1}

Let $\lambda > 0$ and $F$ be a bounded continuous functional on $X$.
Let
\[
\mathcal{U}: [\vr_0, \vm_0, E_0] \in D \mapsto \mathcal{U}[\vr_0, \vm_0, E_0] \in 2^\Omega
\]
be a multivalued mapping having the properties \ref{A1}--\ref{A4}.
Then the map $I_{\lambda,F} \circ \mathcal{U}$ enjoys \ref{A1}--\ref{A4} as well.

\end{Proposition}

\begin{proof}
Apart from the proof of \ref{A2}, we follow the lines of the proof of
Cardona and Kapitanski \cite[Section 2]{CorKap}, which in turn relies on the classical approach by Krylov \cite{KrylNV} for stochastic differential equations.
As a matter of fact, Cardona and Kapitanski \cite{CorKap} consider $\Omega$ as a space of continuous functions on a separable complete metric space $X$. This is not true in our case since due to the possibility of energy sinks the energy $E$ lacks continuity. We therefore present the details of the proof also for reader's convenience.

\begin{itemize}
\item

The map $I_{\lambda, F}: \mathcal{U}[ \vr_0, \vm_0, E_0] \subset\Omega\to R $ is continuous. As the set $\mathcal{U}[ \vr_0, \vm_0, E_0]$ is non--empty and compact,
the set $I_{\lambda, F}\circ\mathcal{U}[\vr_0, \vm_0, E_0]$ is a non--empty compact subset of $\Omega$, which completes the proof of \ref{A1}.

\item
Let $d_H$ be the Hausdorff metric on the subspace $\mathcal{K} \subset 2^\Omega$ of compact sets, specifically,
\[
d_H(K_1, K_2) = \inf_{\ep \geq 0} \left\{ K_1 \subset V_\ep (K_2) \ \mbox{and} \ K_2 \subset V_\ep (K_1) \right\},\ K_1,K_2 \in \mathcal{K},
\]
where $V_\ep(A)$ denotes the $\ep$-neighborhood of a set $A$ in the topology of $\Omega$. To show Borel measurability of the multivalued mapping
\[
[\vr_0, \vm_0, E_0] \in D \mapsto I_{\lambda, F} \circ \mathcal{U}[ \vr_0, \vm_0, E_0] \in \mathcal{K} \subset 2^\Omega,
\]
it is enough to show that the mapping $\mathcal{I}_{\lambda, F}$ defined for any $K \in \mathcal{K}$ as
\[
\mathcal{I}_{\lambda, F} [K] = \left\{ z \in K \ \Big| \ I_{\lambda, F}(z) \leq I_{\lambda, F}(\tilde{z}) \ \mbox{for all}\ \tilde{z} \in K \right\} \in \mathcal{K},
\]
is continuous as a mapping on $\mathcal{K}$ endowed with the Hausdorff metric $d_H$.

Suppose
\[
K_n \stackrel{d_H}{\to} K,\ K_n, K \in \mathcal{K}.
\]
As $I_{\lambda, F}$ is continuous, we easily observe that
\begin{equation} \label{ppp1}
\min_{K_n}
I_{\lambda, F} \to \min_K I_{\lambda,F}.
\end{equation}
Consider the $\ep$-neighborhood $V_\ep (\mathcal{I}_{\lambda, F} [K])$ of the compact set $\mathcal{I}_{\lambda, F} [K]$. Our goal is to show that
\[
\mathcal{I}_{\lambda, F}[K_n] \subset
V_\ep (\mathcal{I}_{\lambda, F} [K]) \ \mbox{for all}\ n \geq n_0(\ep).
\]
Arguing by contradiction, we construct a sequence such that
\[
z_n \in K_n,\  I_{\lambda, F}(z_n) = \min_{K_n}I_{\lambda, F},\ z_n \to z \in K \setminus V_\ep (\mathcal{I}_{\lambda, F} [K]).
\]
Continuity of $I_{\lambda, F}$ yields
\[
I_{\lambda, F}(z_n) \to I_{\lambda, F}(z) > \min_K I_{\lambda, F}
\]
in contrast to \eqref{ppp1}. Interchanging the roles of $K_n$ and $K$ we get the opposite inclusion
\[
\mathcal{I}_{\lambda, F}[K] \subset
V_\ep (\mathcal{I}_{\lambda, F} [K_n]) \ \mbox{for all}\ n \geq n_0(\ep)
\]
by a similar argument. This implies that $\mathcal{I}_{\lambda,F}[K_{n}]\overset{d_{H}}{\to}\mathcal{I}_{\lambda,F}[K]$ and completes the proof of \ref{A2}.

\item
For the shift invariance let us consider some $[\vr, \vm, E]\in I_{\lambda, F}\circ\mathcal{U}[\vr_0, \vm_0, E_0]$ for some  $[\vr_0, \vm_0, E_0]\in D$. We aim to show that the shift $S_T\circ [\vr, \vm, E]$ belongs to the set $I_{\lambda, F}\circ\mathcal{U}[\vr(T), \vm(T), E(T-)]$ for $T>0$ arbitrary. Indeed, for any $[\vr^T, \vm^T, E^T]\in I_{\lambda, F}\circ\mathcal{U}[\vr(T), \vm(T), E(T-)]$ we obtain $$[\vr, \vm, E]\cup_T[\vr^T, \vm^T, E^T]\in \mathcal{U}[\vr_0, \vm_0, E_0]$$
by \ref{A4} and hence
\begin{align*}
I_{\lambda, F}(S_T\circ [\vr, \vm, E])&=\int_0^\infty e^{-\lambda t}F(S_T\circ [\vr, \vm, E](t))\dt\\&=\int_0^\infty e^{-\lambda t} F([\vr, \vm, E](T+t))\dt\\&=e^{\lambda T}\int_T^\infty e^{-\lambda t} F([\vr, \vm, E](t))\dt\\&=e^{\lambda T}\bigg(I_{\lambda, F}[\vr, \vm, E]-\int_0^T e^{-\lambda t} F([\vr, \vm, E](t))\dt\bigg)\\
&\leq e^{\lambda T}\bigg(I_{\lambda, F}([\vr, \vm, E]\cup_T[\vr^T, \vm^T, E^T])-\int_0^T e^{-\lambda t} F([\vr, \vm, E](t))\dt\bigg)\\
&=e^{\lambda T}\int_{T}^{\infty}e^{-\lambda t} F([\vr^{T},\vm^{T},E^{T}](t-T))\,\dt=I_{\lambda, F}[\vr^T, \vm^T, E^T],
\end{align*}
where the inequality follows from the fact that  $[\vr, \vm, E]$ minimizes $I_{\lambda, F}$ on $\mathcal{U}[\vr_{0},\vm_{0},E_{0}]$ by assumption. This implies that $S_T\circ [\vr, \vm, E]$ minimizes
$I_{\lambda, F}$ and consequently belongs to $I_{\lambda, F}\circ\mathcal{U}[\vr(T), \vm(T), E(T-)]$. We have shown property \ref{A3}.

\item

On the other hand, let us consider $[\vr^1, \vm^1, E^1] \in I_{\lambda, F}\circ\mathcal{U} [\vr_0, \vm_0, E_0]$
as well as $[\vr^2, \vm^2, E^2] \in I_{\lambda, F}\circ\mathcal{U} [\vr^1(T), \vm^1(T), E^1(T-)]$ where $[\vr_0, \vm_0, E_0]\in D$ and $T>0$.
We obtain for the continuation $[\vr^1, \vm^1, E^1]\cup_T [\vr^2, \vm^2, E^2]$
that
\begin{align*}
&I_{\lambda, F}([\vr^1, \vm^1, E^1]\cup_T [\vr^2, \vm^2, E^2])\\
&\quad=\int_0^T e^{-\lambda t}F([\vr^1, \vm^1, E^1](t))\dt+\int_T^\infty e^{-\lambda t} F([\vr^2, \vm^2, E^2](t-T))\dt\\
&\quad=\int_0^T e^{-\lambda t}F([\vr^1, \vm^1, E^1](t))\dt+e^{-\lambda T}I_{\lambda,F}[\vr^2, \vm^2, E^2]\\
&\quad\leq \int_0^T e^{-\lambda t}F([\vr^1, \vm^1, E^1](t))\dt+e^{-\lambda T}I_{\lambda,F}(S_T\circ[\vr^1, \vm^1, E^1])\\
&\quad=I_{\lambda, F}[\vr^1, \vm^1, E^1],
\end{align*}
where the inequality follows from the fact that  $[\vr^2, \vm^2, E^2]$ is a minimizer of $I_{\lambda,F}$ in the set $\mathcal{U} [\vr^1(T), \vm^1(T), E^1(T-)]$. As $[\vr^1, \vm^1, E^1]$ is a minimizer in $\mathcal{U} [\vr_0, \vm_0, E_0]$ and $[\vr^1, \vm^1, E^1]\cup_T [\vr^2, \vm^2, E^2]\in\mathcal{U} [\vr_0, \vm_0, E_0]$ by \ref{A4} we must have equality and $[\vr^1, \vm^1, E^1]\cup_T [\vr^2, \vm^2, E^2]$ is a minimizer too. This proves \ref{A4} for $I_{\lambda, F}\circ\mathcal{U}$ and the proof is complete.
\end{itemize}
\end{proof}

\subsection{Selection sequence}

The first step is to select only those solutions that are admissible, meaning minimal with respect to the relation
$\prec$ introduced in Definition \ref{DD2}. To this end, we consider the functional $I_{1, \beta}$ with
\[
\beta(\vr, \vm, E) = \beta (E), \ \beta: R \to R \ \mbox{smooth, bounded, and strictly increasing.}
\]

\begin{Lemma} \label{CL2}

Suppose that $[\vr, \vm, E] \in \mathcal{U}[\vr_0, \vm_0, E_0]$ satisfies
\[
\int_0^\infty \exp(-t) \beta( E(t) ) \, \dt  \leq
\int_0^\infty \exp(-t) \beta( \tilde{E}(t) ) \, \dt
\]
for any $[\tvr, \tvm, \tilde{E}] \in \mathcal{U}[\vr_0, \vm_0, E_0]$.
Then $[\vr, \vm, E]$ is $\prec$ minimal, meaning, admissible.

\end{Lemma}
\begin{proof}
We proceed by contradiction.
Let $[\tvr, \tvm, \tilde{E}] \in \mathcal{U}[\vr_0, \vm_0, E_0]$ be such that
$
[\tvr, \tvm, \tilde{E}] \prec [\vr, \vm, E],
$ that is,
$
\tilde{E} \leq E
$ in $(0,\infty).$
Then we get
\[
\beta(E) \geq \beta(\tilde{E}), \ \mbox{and}\ \int_0^\infty \exp(-t) \left[ \beta(E) - \beta(\tilde{E}) \right] \dt \leq 0;
\]
whence $E = \tilde E$ a.a. in $(0,\infty)$ since $\beta$ is strictly increasing.
\end{proof}

\emph{Proof of Theorem \ref{ST1}.}
Selecting $I_{1,\beta} \circ \mathcal{U}$ from $\mathcal{U}$ we know that the new selection contains only admissible solutions
(minimal with respect to $\prec$) for any $[\vr_0, \vm_0, E_0] \in D$.

Next, we choose a countable basis $\{ e_n \}_{n=1}^\infty$ in $L^2(\Q)$ formed by trigonometric polynomials, its vector valued analogue $\{ \vc{w}_m \}_{m=1}^\infty$ in $L^2(\Q; R^N)$, and a countable set $\{\lambda_k \}_{k=1}^{\infty}$ which is dense in $(0,\infty)$. We consider a countable family of functionals,
\[
\begin{split}
I_{k,0,0}[\vr, \vm, E] &= \int_0^\infty \exp(-\lambda_{k}t) \beta(E(t)) \dt,\\
I_{k,n,0}[\vr, \vm, E] &= \int_0^\infty \exp(-\lambda_{k}t) \beta \left( \intQ{ \vr e_n } \right) \dt,\\
\ I_{k,0,m} [\vr, \vm, E] &= \int_0^\infty \exp(-\lambda_{k}t) \beta \left( \intQ{ \vm \cdot \vc{w}_m } \right) \dt,
\end{split}
\]
and let $\{(k(j),n(j),m(j))\}_{j=1}^{\infty}$ be an enumeration of all the involved combinations of indices, that is, an enumeration of the countable set $$(\mathbb{N}\times\{0\}\times\{0\})\cup (\mathbb{N}\times\mathbb{N}\times\{0\})\cup( \mathbb{N}\times\{0\}\times\mathbb{N} ).$$
We define
\[
\mathcal{U}^j = I_{k(j), n(j), m(j)} \circ \dots \circ I_{k(1),n(1),m(1)} \circ I_{1,\beta} \circ \mathcal{U}, \ j=1,2,\dots,
\]
and
\[
\mathcal{U}^\infty = \cap_{j=1}^\infty \mathcal{U}^j.
\]
By Proposition \ref{AP1} the set--valued mapping
\[
D\ni [\vr_0, \vm_0, E_0]  \mapsto \mathcal{U}^\infty [\vr_0, \vm_0, E_0]
\]
enjoys the properties \ref{A1}--\ref{A4}. Indeed, since $\mathcal{U}^{\infty}[\vr_{0},\vm_{0},E_{0}]$  is an intersection of countably many non--empty compact nested sets, it is non--empty  and compact. As it is an intersection set--valued map obtained from measurable set--valued maps,  it also measurable. The shift property \ref{A3} as well as the continuation property \ref{A4} are straightforward.

Finally, we claim that for every $[\vr_{0},\vm_{0},E_{0}]\in D$ the set $\mathcal{U}^\infty[\vr_{0},\vm_{0},E_{0}]$ is a singleton, meaning there exists a single trajectory $U\{\vr_{0},\vm_{0},E_{0}\}\in \Omega$ such that
\begin{equation}\label{eq:ss}
\mathcal{U}^\infty [\vr_0, \vm_0, E_0] = \big\{U \left\{ \vr_0, \vm_0, E_0 \right\}\big\} 
\end{equation}
for any $[\vr_0, \vm_0, E_0] \in D$, which completes the proof of Theorem \ref{ST1}. Indeed, the semigroup property follows from the definition of the shift property  \ref{A3}: for all $t_{1},t_{2}\geq 0$ it holds
$$U\{\vr_{0},\vm_{0},E_{0}\}(t_{1}+t_{2})=S_{t_{1}}\circ U\{\vr_{0},\vm_{0},E_{0}\}(t_{2})=U\{U\{\vr_{0},\vm_{0},E_{0}\}(t_{1}-)\}(t_{2}).$$
To verify \eqref{eq:ss}, we observe that
\[
\begin{split}
I_{k(j), n(j), m(j)} &[\vr^1, \vm^1, E^1] =
I_{k(j), n(j), m(j)} [\vr^2, \vm^2, E^2] \\ &\mbox{for any} \ [\vr^1, \vm^1, E^{1}], \, [\vr^2, \vm^2, E^{2}] \in
\mathcal{U}^\infty[\vr_0, \vm_0, E_0]
\end{split}
\]
for all $j = 1,2,\dots$.
This implies, by means of Lerch's theorem and the choice of $\{(k(j),n(j),m(j))\}_{j=1}^{\infty}$ that
\begin{align*}
 \beta(E^1(t))&=\beta(E^2(t)),\\
 \beta \left( \intQ{ \vr^1 e_n } \right)&=\beta \left( \intQ{ \vr^2 e_n } \right),\\
 \beta \left( \intQ{ \vm^1 \cdot \vc{w}_m } \right)&=\beta \left( \intQ{ \vm^2 \cdot \vc{w}_m } \right),
\end{align*}
for all $m\in\mathbb N$ and a.a. $t\in (0,\infty)$.
As $\beta$ is strictly increasing and $\{ e_n \}_{n=1}^\infty$ and $\{ \vc{w}_m \}_{m=1}^\infty$ form a basis in $L^2(\Q)$ and $L^2(\Q; R^N)$ resepectively we conclude
\[
\vr^1 = \vr^2,\ \vm^1 = \vm^2, \ \mbox{and}\ E^1 = E^2 \ \mbox{a.a. on}\ (0, \infty).
\]
which finishes the proof. \hfill $\Box$

\section{Concluding remarks}
\label{C}

Regularity of the constructed semiflow as well as possible dependence of the trajectories on the initial data represent  major open issues  that probably cannot be solved within the present abstract framework. In what follows, we discuss some simple observations
that may shed some light on the complexity of the problem.

\subsection{Energy profile}

The hypothetical possibility of ``energetic sinks'' - the times $T > 0$ for which
\[
\intQ{ \left[ \frac{|\vm|^2}{\vr} + \frac{a}{\gamma - 1} \vr^\gamma \right] (T, \cdot) } <
E(T+)
\]
implies the existence of solutions in the semiflow with positive jump of the initial energy:
\[
\intQ{ \left[ \frac{|\vm_0|^2}{\vr_0} + \frac{a}{\gamma - 1} \vr_0^\gamma \right]  } < E(0+).
\]
It is interesting to note that the existence proof presented in Proposition \ref{ESP1} does not provide solutions of this type.
One may be tempted to say that these are exactly the solutions obtained via the method of convex integration, however, such a conclusion
is not straightforward as shown in the next section.

\subsection{Wild weak solutions}

In the context of the recent results achieved by the method of convex integration, see \cite{Chiod}, \cite{DelSze3}, \cite{Fei2016},
some of the solutions involved in the semiflow might be the so--called wild (weak) solutions producing energy. This seems particularly relevant for the initial data of the form
\begin{equation*}
\vr_0, \ \vm_0, \ \mbox{with}\ E(0+) > \intQ{ \left[ \frac{1}{2} \frac{|\vm_0|^2}{\vr_0} + \frac{a}{\gamma - 1} \vr_0^\gamma \right]
}.
\end{equation*}
However, such a possibility seems to be ruled out by the available convex integration ansatz used in the context of compressible flow,
cf. \cite{Fei2016}.
Indeed the weak solutions are ``constructed'' with \emph{prescribed} energy profile
$e_{\rm kin}(t,x) + e_{\rm int}(t,x)$ - a given continuous function of $t$ and $x$ -
as limits of subsolutions $[\vr^s, \vm^s]$. The subsolutions
$\vr^s$, $\vm^s$ satisfy the strict inequality
\[
\left[
\frac{1}{2} \frac{|\vm^s|^2}{\vr^s} + \frac{a}{\gamma - 1} (\vr^s)^\gamma \right](t,x)
< e_{\rm kin}(t,x) + e_{\rm int}(t,x)
\ \mbox{for any}\ t > 0,\ x \in \Q.
\]
Consequently, the same method gives rise to another solution with the same initial data with a chosen energy profile
\[
\tilde{e}_{\rm kin}(t,x) + \tilde{e}_{\rm int}(t,x) < e_{\rm kin}(t,x) + e_{\rm int}(t,x),\ t > 0,\ x \in \Q,
\]
which rules out the former solution on the basis of the $\prec$ minimality.

\subsection{Total mass conservation and stability of equilibrium states}

It follows directly from the continuity equation \eqref{W4} that any dissipative solution conserves the total mass,
\begin{equation} \label{CC20}
\intQ{ \vr(\tau, \cdot) } = \intQ{ \vr_0 } = M \ \mbox{for any}\ \tau \geq 0.
\end{equation}
The equilibrium states
\[
\vr_M \equiv \frac{M}{|T^N|}\geq0, \ \vm_M \equiv 0, \ E_M \equiv \frac{a}{\gamma - 1} \intQ{ \vr^\gamma_M },
\]
are global in time regular solutions; whence, in accordance with the weak--strong uniqueness principle stated in
Proposition \ref{WP1},
\[
U\left\{ \vr_M, \vm_M = 0, E_M \right\} = [\vr_M, 0, E_M] \ \mbox{for any}\ M \geq 0.
\]
We claim that
\[
U\left\{ \vr_M, \vm_M = 0, E_0 \right\} = [\vr_M, 0, E_M]\ \mbox{for any}\ E_0 > E_M,
\]
meaning the energy cannot ``jump up'' for any dissipative solution in the selection starting from the equilibrium $[\vr_M, \vm_M]$. Indeed suppose
that
\[
[\vr, \vm, E] \in \mathcal{U}^\infty[\vr_M, 0 , E_0],\ E_0 > E_M.
\]
In accordance with \eqref{CC20}, the total mass is conserved, namely
\begin{equation} \label{CC23}
\intQ{ \vr(\tau, \cdot) } = \intQ{ \vr_M } = M \ \mbox{for any}\ \tau \geq 0.
\end{equation}
On the other hand, the energy is  weakly lower semi--continuous, whence
\[
\frac{a}{\gamma - 1} \intQ{ \vr^\gamma (\tau, \cdot) } \leq
\intQ{ \left[ \frac{1}{2} \frac{ |\vm|^2 }{\vr} + \frac{a}{\gamma - 1} \vr^\gamma \right] (\tau, \cdot) }
\leq E(\tau \pm)
\]
for any $\tau > 0$.
Finally, we use \eqref{CC23} and Jensen's inequality to obtain
\begin{align*}
\frac{1}{|\Q|} \intQ{ \vr_M^\gamma } &= \left( \frac{1}{|\Q|}  \intQ{ \vr_M } \right)^\gamma =
\left( \frac{1}{|\Q|}  \intQ{ \vr } \right)^\gamma \\&\leq \frac{1}{|\Q|} \intQ{ \vr^\gamma },
\end{align*}
where the equality holds if and only if $\vr = \vr_M$. Consequently $E(\tau \pm) \geq E_M$ for any $\tau > 0$, meaning
$[\vr, \vm, E_0]$ can be $\prec$ minimal if and only if $\vr = \vr_M$, $\vm = 0$.

We have obtained the following
\begin{Corollary} \label{COC1}

Let $[\vr, \vm, E] = U\{ \vr_0, \vm_0, E_0 \}$ belong to the semiflow constructed in Theorem~\ref{ST1}.
Suppose that
\[
\vr(T, \cdot) = \vr_M, \ \vm(T, \cdot) = 0 \ \mbox{for some}\ T \geq 0.
\]

Then
\[
\vr(\tau, \cdot) = \vr_M, \ \vm(\tau, \cdot) = 0 \ \mbox{for all}\ \tau \geq T.
\]

\end{Corollary}

\subsection{General {equation of state}}

The results presented above can be extended in a straightforward manner to a more general barotropic {equation of state} provided the pressure
$p = p(\vr)$ and the pressure potential $P(\vr)$ given by
\[
P'(\vr) \vr - P(\vr) = p(\vr),
\]
satisfy the asymptotic ``adiabatic law''
\[
p'(\vr) > 0 \ \mbox{for}\ \vr > 0,\
\lim_{\vr \to \infty} \frac{p(\vr)}{P(\vr)} = \gamma - 1 , \ \mbox{with}\ \gamma > 1.
\]
If $\gamma = 1$, we need an extra hypothesis
\[
\liminf_{\vr \to \infty} p'(\vr) > 0.
\]

\subsection{Relative energy inequality}
\label{REI}

Let $P$ be the pressure potential introduced in the previous section. We define the \emph{relative energy},
\[
\mathcal{E} \left( \vr, \vm \ \Big| r, \vc{U} \right) = 
\frac{1}{2} \vr \left| \frac{\vm}{\vr} - \vc{U} \right|^2 + P(\vr) - P'(r)(\vr - r) - P(r).  
\]
Following \cite{GSWW} we can derive the relative energy inequality 
\begin{equation*} 
\begin{split}
&\intQ{ \mathcal{E} \left( \vr, \vm \ \Big| r, \vc{U} \right) (\tau, \cdot) } \\ &\leq 
 \left( E(0+) - \intQ{ \left[ \frac{|\vm_0|^2}{\vr_0} + P(\vr_0) \right] } \right)  + \intQ{ \mathcal{E} \left( \vr_0, \vm_0 \ \Big| r(0, \cdot) , \vc{U} (0, \cdot) \right)  }
 \\
&\quad+ \int_0^\tau \intQ{ \frac{1}{r} \Big(r \left( \partial_t \vc{U} +  \vc{U} \cdot \Grad \vc{U} \right) + \Grad p(r) \Big) \Big( \vr \vc{U} - \vm \Big) } \dt \\
&\quad+  \int_0^\tau \intQ{ P''(r) (r - \vr) \Big( \partial_t r + \Div (r \vc{U}) \Big) } \dt \\
&\quad +c\int_0^\tau \| \Grad \vc{U} \|_{L^\infty(\Q)}\intQ{\mathcal{E} \left( \vr, \vm \ \Big| r, \vc{U} \right)}\dt
\end{split}
\end{equation*}
that holds for any dissipative solution $[\vr, \vm, E]$ starting from the initial data $[\vr_0, \vm_0, E_0] \in D$, and any
$r \in W^{1,\infty}_{\rm loc}([0,\infty)\times\Q)$, $\vc{U} \in W^{1,\infty}_{\rm loc}([0,\infty)\times\Q; R^N)$, $r > 0$. In particular, we have by Gronwall's lemma
\[
\begin{split}
&\intQ{ \mathcal{E} \left( \vr, \vm \ \Big| r, \vc{U} \right) (\tau, \cdot) } \\ &\leq
\left[ \left( E(0+) - \intQ{ \left[ \frac{|\vm_0|^2}{\vr_0} + P(\vr_0) \right] } \right)  + \intQ{ \mathcal{E} \left( \vr_0, \vm_0 \ \Big| r(0, \cdot) , \vc{U} (0, \cdot )\right)  }
\right] \\
&\times 
\exp \left( c \int_0^\tau \| \Grad \vc{U} \|_{L^\infty(\Q)} \ \dt \right)
\end{split}
\]
for any strong solution $r, \vc{M} = r \vc{U}$, $r > 0$, of the Euler system, which yields the weak--strong
uniqueness property stated in Proposition \ref{WP1}.\\\

\centerline{\bf Declaration}
\noindent{
The authors declare that there are no conflicts of interest.}

\def\cprime{$'$} \def\ocirc#1{\ifmmode\setbox0=\hbox{$#1$}\dimen0=\ht0
  \advance\dimen0 by1pt\rlap{\hbox to\wd0{\hss\raise\dimen0
  \hbox{\hskip.2em$\scriptscriptstyle\circ$}\hss}}#1\else {\accent"17 #1}\fi}

\end{document}